\documentclass[a4paper,10pt,leqno]{amsart}
\usepackage[cp1250]{inputenc}
\usepackage{ifthen}
\usepackage{amssymb}
\usepackage{amsmath}
\usepackage{amsfonts}
\usepackage{latexsym}
\usepackage{hyperref}

\numberwithin{equation}{section}
\newtheorem{theorem}{Theorem}[section]
\theoremstyle{plain}
\newtheorem{acknowledgement}{Acknowledgement}

\newtheorem{corollary}[theorem]{Corollary}

\newtheorem{lemma}{Lemma}[section]

\numberwithin{equation}{section}

\begin{document}
\title[Fekete-Szegö Problem for Generalized Bi-Subordinate Functions of
Complex Order]{Fekete-Szegö Problem for Generalized Bi-Subordinate Functions
of Complex Order}
\author{\textbf{\ Sercan Topkaya}}
\address{(S. Topkaya and E. Deniz), \textit{Department of Mathematics$,$
Faculty of Science} \textit{and Letters,} \textit{Kafkas University$,$ Kars$%
, $ TURKEY} }
\email{topkaya.sercan@hotmail.com(S. Topkaya) and edeniz36@gmail.com (E.
Deniz)}
\author{\textbf{Erhan Deniz}}

\begin{abstract}
In this paper, we obtain Fekete-Szegö inequality for the generalized
bi-subordinate functions of complex order. The results, which are presented
in this study, would generalize those in related works of several earlier
authors.
\end{abstract}

\maketitle

\section{\textbf{INTRODUCTION}}

Let $\mathcal{A}$ be the class of analytic functions in the open unit disk $%
\mathbb{D}=\{z\in \mathbb{%
\mathbb{C}
}:$ $\left\vert z\right\vert <1\}$ and let $\mathcal{S}$ be the class of
functions $f$ that are analytic and univalent in $\mathbb{D}$ and are of the
form%
\begin{equation}
f\left( z\right) =z+\sum_{k=2}^{\infty }a_{k}z^{k}.  \label{1}
\end{equation}

A function $f\in \mathcal{A}$ is said to be subordinate to a function $g\in 
\mathcal{A},$ denoted by $f\prec g,$ if there exists a function $w$ $\in 
\mathcal{B}_{0}$ where 
\begin{equation*}
\mathcal{B}_{0}:=\left\{ w\in \mathcal{A}:w\left( 0\right) =0,\text{ }%
\left\vert w\left( z\right) \right\vert <1,\text{ \ \ }\left( z\in \mathbb{D}%
\right) \right\} ,
\end{equation*}%
such that 
\begin{equation*}
f(z)=g(w(z)),\text{ \ \ }\left( z\in \mathbb{D}\right) .
\end{equation*}

We let $\mathcal{S}^{\ast }$ consist of starlike functions $f\in \mathcal{A}$%
, that is $\Re\{zf^{\prime }(z)\diagup f(z)\}>0$ in $\mathbb{D}$ and $%
\mathcal{C}$ consist of convex functions $f\in \mathcal{A}$, that is $1+%
\Re\{zf^{\prime \prime }(z)\diagup f^{\prime }(z)\}>0$ in $\mathbb{D}$. In
terms of subordination, these conditions are, respectively, equivalent to \ 
\begin{equation*}
\mathcal{S}^{\ast }\equiv \left\{ f\in \mathcal{A}:\frac{zf^{\prime }(z)}{%
f(z)}\prec \frac{1+z}{1-z}\right\}
\end{equation*}%
and%
\begin{equation*}
\text{\ }\mathcal{C}\equiv \left\{ f\in \mathcal{A}:1+\frac{zf^{\prime
\prime }(z)}{f^{\prime }(z)}\prec \frac{1+z}{1-z}\right\} .
\end{equation*}%
A generalization of the above two classes, according to Ma and Minda \cite%
{Ma}, are%
\begin{equation*}
\text{\ }\mathcal{S}^{\ast }(\varphi )\equiv \left\{ f\in \mathcal{A}:\text{ 
}\frac{zf^{\prime }(z)}{f(z)}\prec \varphi (z)\right\}
\end{equation*}%
and%
\begin{equation*}
\text{\ }\mathcal{C}(\varphi )\equiv \left\{ f\in \mathcal{A}:\text{ }1+%
\frac{zf^{\prime \prime }(z)}{f^{\prime }(z)}\prec \varphi (z)\right\}
\end{equation*}%
where $\varphi $ \ is a positive real part function normalized by $\varphi
(0)=1$, $\varphi ^{\prime }(0)>0$ and $\varphi $ maps $D$ onto a region
starlike with respect to $1$ and symmetric with respect to the real axis.
Obvious extensions of the above two classes (see \cite{Ra}) are%
\begin{equation*}
\mathcal{S}^{\ast }(\gamma ;\varphi )\equiv \left\{ f\in A:1+\frac{1}{\gamma 
}\left( \frac{zf^{\prime }(z)}{f(z)}-1\right) \prec \varphi (z);\text{ \ \ }%
\gamma \in 
\mathbb{C}
\backslash \left\{ 0\right\} \right\}
\end{equation*}%
and%
\begin{equation*}
\text{\ }C(\gamma ;\varphi )\equiv \left\{ f\in A:1+\frac{1}{\gamma }\left( 
\frac{zf^{\prime \prime }(z)}{f^{\prime }(z)}\right) \prec \varphi (z);\text{
\ \ }\gamma \in 
\mathbb{C}
\backslash \left\{ 0\right\} \right\} .
\end{equation*}%
In literature, the functions belonging to these classes are called Ma-Minda s%
\textit{tarlike and convex of complex order} $\gamma $ $\left( \gamma \in 
\mathbb{C}
\diagdown \{0\}\right) $, respectively.

Some of the special cases of the above two classes \ $S^{\ast }(\gamma
;\varphi )$\ and $C(\gamma ;\varphi )$\ are
\begin{itemize}
\item[(1)] $\mathcal{S}^{\ast }\left( (1,(1+Az)\diagup (1+Bz)\right) =\mathcal{S}%
[A,B]$ and $\mathcal{C}(1,(1+Az)\diagup (1+Bz))=\mathcal{C}[A,B],$ $\left(
-1\leq B<A\leq 1\right) $ are classes of Janowski starlike and convex
functions, respectively,

\item[(2)] $\mathcal{S}^{\ast }((1-\beta )e^{-i\delta }\cos \delta ,(1+z)\diagup
(1-z))=\mathcal{S}^{\ast }[\delta ,\beta ]$ and $\mathcal{C}((1-\beta
)e^{-i\delta }\cos \delta ,(1+z)\diagup (1-z))=\mathcal{C}[\delta ,\beta ],$ 
$\left( \left\vert \delta \right\vert <\pi \diagup 2,\text{ }0\leq \beta
<1\right) $ are classes of $\delta $-spirallike and $\delta $-Robertson
univalent functions of order $\beta ,$ respectively,

\item[(3)] $\mathcal{S}^{\ast }(1,(1+(1-2\beta )z)\diagup (1-z))=\mathcal{S}^{\ast
}(\beta )$ and $\mathcal{C}\left( 1,(1+(1-2\beta )z)\right) \diagup (1-z))=%
\mathcal{C}(\beta )$ $\left( 0\leq \beta <1\right) $ are classes of starlike
and convex functions of order $\beta ,$ respectively,

\item[(4)] $\mathcal{S}^{\ast }\left( 1,\left( \frac{1+z}{1-z}\right) ^{\beta
}\right) =\mathcal{S}_{\beta }^{\ast }$ and $\mathcal{C}\left( 1,\left( 
\frac{1+z}{1-z}\right) ^{\beta }\right) =\mathcal{C}_{\beta }$ are class of
strongly starlike and convex functions of order $\beta ,$ respectively,

(\item[(5)] $\mathcal{S}^{\ast }\left( 1,\sqrt{1+z}\right) =\mathcal{S}_{L}^{\ast
}=\left\{ f\in \mathcal{A}:\left\vert \left( \frac{zf^{\prime }(z)}{f(z)}%
\right) ^{2}-1\right\vert <1\right\} $ is class of lemniscate starlike
functions,

\item[(6)] $\mathcal{S}^{\ast }(\gamma ,(1+z)\diagup (1-z))=\mathcal{S}^{\ast
}[\gamma ]$ and $\mathcal{C}(\gamma ,(1+z)\diagup (1-z))=\mathcal{C}[\gamma
] $ $\left( \gamma \in 
\mathbb{C}
\diagdown \{0\}\right) $ are classes of starlike and convex functions of
complex order, respectively,

\item[(7)] $\mathcal{S}^{\ast }\left( 1,q_{k}(z)\right) =k-\mathcal{S}_{P}^{\ast
}=\left\{ f\in \mathcal{A}: \Re \left( \frac{zf^{\prime }(z)}{f(z)}%
\right) >k\left\vert \frac{zf^{\prime }(z)}{f(z)}-1\right\vert \right\} $ is
class of $k-$parabolik starlike functions,

\item[(8)] $\mathcal{C}\left( 1,q_{k}(z)\right) =k-\mathcal{UCV}=\left\{ f\in 
\mathcal{A}:\Re \left( 1+\frac{zf^{\prime \prime }(z)}{f^{\prime }(z)}%
\right) >k\left\vert \frac{zf^{\prime \prime }(z)}{f^{\prime }(z)}%
\right\vert \right\} $ is class of $k-$uniformly convex functions.
\end{itemize}
Here, for $0\leq k<\infty $ the function $q_{k}:$ $\mathbb{D\rightarrow \{}%
w=u+iv\in 
\mathbb{C}
:$ $\ u^{2}>k^{2}\left( (u-1)^{2}+v^{2}\right) ,u>0\mathbb{\}}$ has the form 
$q_{k}(z)=1+Q_{1}z+Q_{2}z^{2}+...,$\ $\left( {z\in }\mathbb{D}\right) $ where%
\begin{equation}
Q_{1}=\left\{ {%
\begin{array}{l}
\frac{2\mathcal{B}^{2}}{1-k^{2}};\;\;\;\;\;\;\;\;\;\;\;\;\ \ \ \ 0\leqslant
k<1, \\ 
\frac{8}{\pi ^{2}};\;\ \;\;\;\;\;\;\;\;\;\;\;\;\;\;\ \ \ \ \ \ \ k=1, \\ 
\frac{\pi ^{2}}{4(k^{2}-1)\sqrt{t}(1+t)\mathcal{K}^{2}(t)};\text{ \ \ \ }%
k>1,\;\;%
\end{array}%
}\right. ,\text{ \ }Q_{2}=\left\{ {%
\begin{array}{l}
\frac{(\mathcal{B}^{2}+2)}{3}Q_{1};\;\;\;\;\;\;\;\;\;\;\;\;\ \ \ \
0\leqslant k<1, \\ 
\frac{2}{3}Q_{1};\;\ \;\;\;\;\;\;\;\;\;\;\;\;\;\;\ \ \ \ \ \ \ \ \ k=1, \\ 
\frac{\left[ {4\mathcal{K}^{2}(t)(t^{2}+6t+1)-\pi ^{2}}\right] }{24\sqrt{t}%
(1+t)\mathcal{K}^{2}(t)}Q_{1};\text{ \ \ \ }k>1,\;\;%
\end{array}%
}\right.  \label{eq171}
\end{equation} 
\noindent with $\mathcal{B}=\frac{2}{\pi }\arccos k$ \noindent and $\mathcal{%
K}(t)$ is the complete elliptic integral of first kind (see \cite{Ka}).

A function $f\in \mathcal{A}$ is said to be bi-univalent in $\mathbb{D}$ if
both $f$ and its inverse map $f^{-1}$ are univalent in $\mathbb{D}.$ Let $%
\sigma $ be the class functions $f\in \mathcal{S}$ that are bi-univalent in $%
\mathbb{D}$. For a brief history and interesting examples of functions which
are in (or are not in) in the class $\sigma $, including various properties
of such functions we refer the reader to the work of Srivastava et al. \cite%
{Sri} and references therein. Bounds for the first few coefficients of
various subclasses of bi-univalent functions were obtained by a variety of
authors including \cite{R,CaOrNi,Er1}, \cite{Fra}, \cite{Ku}, \cite%
{Sri2,HM2,Qing1,Qing2}. Not much was known about the bounds of the general
coefficients $a_{n};n\geq 4$\ of subclasses of $\sigma $ up until the
publication of the article \cite{JH1} by Jahangiri and Hamidi and followed
by a number of related publications. Moreover, many author have considered
the Fekete-Szegö problem for various subclasses of $\mathcal{A}$, the upper
bound for $\left\vert a_{3}-\mu a_{2}^{2}\right\vert $ is investigated by
many different authors $\left( \text{see}\left[ \text{\cite{HM3}},\text{\cite%
{zapr14}},\text{\cite{hoafr}},\text{\cite{hem12}}\right] \right) .$\ In this
paper, we apply the Fekete-Szegö inequality to certain subclass of
generalized bi-subordinate functions of complex order.

\section{\textbf{Coefficient Estimates}}

In the sequel, it is assumed that $\varphi $ is an analytic function with
positive real part in the unit disk $\mathbb{D}$, satisfying $\varphi (0)=1,$
$\varphi ^{\prime }(0)>0,$ and $\varphi (\mathbb{D})$ is symmetric with
respect to the real axis. Such a function is known to be typically real with
the series expansion $\varphi (z)=1+B_{1}z+B_{2}z^{2}+B_{3}z^{3}+...$ where $%
B_{1}$, $B_{2}$ are real and\ $B_{1}>0.$\ Motivated by a class of functions
defined by the first author \cite{Er1}, we define the following
comprehensive class of analytic functions 
\begin{equation*}
\text{\ }\mathcal{S}(\lambda ,\gamma ;\varphi )\equiv \left\{ f\in A:1+\frac{%
1}{\gamma }\left( \frac{zf^{\prime }(z)+\lambda z^{2}f^{\prime \prime }(z)}{%
(1-\lambda )f(z)-\lambda zf^{\prime }(z)}-1\right) \prec \varphi (z)\text{; }%
0\leq \lambda \leq 1\text{, }\gamma \in 
\mathbb{C}
\backslash \left\{ 0\right\} .\right\}
\end{equation*}

A function $f\in \mathcal{A}$ is siad the be generalized bi-subordinate of
complex order $\gamma $\ and type $\lambda $ if both $f$\ and its inverse
map $g=f^{-1}$\ are in \ $\mathcal{S}(\lambda ,\gamma ;\varphi )$. As
special cases of the class \ $\mathcal{S}(\lambda ,\gamma ;\varphi )$\ we
have \ $\mathcal{S}(0,\gamma ;\varphi )\equiv $\ $\mathcal{S}^{\ast }\left(
\gamma ;\varphi \right) $ and $\mathcal{S}\left( 1,\gamma ;\varphi \right)
\equiv $\ $\mathcal{C}\left( \gamma ;\varphi \right) $.

To prove our next theorems, we shall need the following well-known lemma
(see \cite{keog}).

\begin{lemma}
\label{l1} (see \cite{keog})Let the function $w\in \mathcal{B}_{0}$ be given
by%
\begin{equation*}
w(z)=c_{1}z+c_{2}z^{2}+\cdot \cdot \cdot \left( z\in \mathbb{D}\right) ,
\end{equation*}%
then for by every complex number $s,$%
\begin{equation*}
\left\vert c_{2}-sc_{1}^{2}\right\vert \leq 1+\left( \left\vert s\right\vert
-1\right) \left\vert c_{1}\right\vert ^{2}.
\end{equation*}
\end{lemma}

In the following theorem, we consider fuctional $\left\vert a_{3}-\mu
a_{2}^{2}\right\vert $ for $\gamma $ nonzero complex number and $\mu \in 
\mathbb{C}
.$

\begin{theorem}
\label{T1}\textit{Let }$\gamma $ be a nonzero complex number, $\mu \in 
\mathbb{C}
$ \ and $0\leq \lambda \leq 1.$ If both functions $f$ of the form (\ref{1})
and its inverse maps $g=f^{-1}$ are in $\mathcal{S}(\lambda ,\gamma ;\varphi
),$ then we obtain,%
\begin{equation}
\left\vert a_{2}\right\vert \leq \frac{\left\vert \gamma \right\vert B_{1}}{%
\left( 1+\lambda \right) },  \label{t11}
\end{equation}%
\begin{equation}
\left\vert a_{3}\right\vert \leq \frac{\left\vert \gamma \right\vert
\left\vert B_{1}\right\vert }{4\left( 1+2\lambda \right) }\max \{2,\left(
\left\vert s\right\vert +\left\vert t\right\vert \right) \}  \label{t12}
\end{equation}%
and%
\begin{equation}
\left\vert a_{3}-\mu a_{2}^{2}\right\vert \leq \left\{ 
\begin{array}{c}
\frac{B_{1}\left\vert \gamma \right\vert }{2\left( 1+2\lambda \right) }\text{%
\ \ \ \ \ \ \ \ \ \ \ \ \ \ \ \ if \ }\mathcal{L} \leq 2, \\ 
\frac{B_{1}\left\vert \gamma \right\vert }{4\left( 1+2\lambda \right) }
{L} \text{\ \ \ \ \ \ \ \ \ \ \ \ \ if \ }\mathcal{L} \geq 2,
\end{array}%
\right.  \label{t13}
\end{equation}%
where $s=\frac{B_{2}}{B_{1}}-\frac{4B_{1}\gamma \left( 1+2\lambda \right) }{%
\left( 1+\lambda \right) ^{2}},$ $t=\frac{B_{2}}{B_{1}}$ and $\mathcal{%
{L} =}\left\vert \frac{B_{2}}{B_{1}}+\left( 1-\mu \right) \frac{%
4B_{1}\gamma \left( 1+2\lambda \right) }{\left( 1+\lambda \right) ^{2}}%
\right\vert +\left\vert \frac{B_{2}}{B_{1}}\right\vert .$
\end{theorem}
\begin{proof}
\label{P1}Let $f(z)\in \mathcal{S}(\lambda ,\gamma ;\varphi )$ and $g=f^{-1}$%
. Then there are two functions $u(z)=c_{1}z+c_{2}z^{2}+\cdot \cdot \cdot
\left( z\in \mathbb{D}\right) $ and $v(w)=d_{1}w+d_{2}w^{2}+\cdot \cdot
\cdot ,$ such that 
\begin{equation*}
1+\frac{1}{\gamma }\left( \frac{zf^{\prime }(z)+\lambda z^{2}f^{\prime
\prime }(z)}{(1-\lambda )f(z)-\lambda zf^{\prime }(z)}-1\right) =\varphi
(u(z))\text{\ }
\end{equation*}%
\begin{equation}
=1+\frac{\left( 1+\lambda \right) a_{2}}{\gamma }z+\left[ \frac{2\left(
1+2\lambda \right) a_{3}-\left( 1-\lambda \right) ^{2}a_{2}^{2}}{\gamma }%
\right] z^{2}+\cdot \cdot \cdot  \label{p11}
\end{equation}%
and 
\begin{equation*}
1+\frac{1}{\gamma }\left( \frac{wg^{\prime }(w)+\lambda w^{2}g^{\prime
\prime }(w)}{(1-\lambda )g(w)-\lambda wg^{\prime }(w)}-1\right) =\varphi
(v(w))
\end{equation*}%
\begin{equation}
=1-\frac{\left( 1+\lambda \right) a_{2}}{\gamma }w+\left[ \frac{-2\left(
1+2\lambda \right) a_{3}+\left( 3+6\lambda -\lambda ^{2}\right) a_{2}^{2}}{%
\gamma }\right] w^{2}+\cdot \cdot \cdot  \label{p12}
\end{equation}%
Equating coefficients of both side of equations (\ref{p11}) and (\ref{p12})
yield%
\begin{eqnarray}
\frac{\left( 1+\lambda \right) a_{2}}{\gamma } &=&B_{1}c_{1},\text{ \ \ \ \ }%
\frac{2\left( 1+2\lambda \right) a_{3}-\left( 1-\lambda \right) ^{2}a_{2}^{2}%
}{\gamma }=B_{1}c_{2}+B_{2}c_{1}^{2},  \label{p13} \\
\frac{-\left( 1+\lambda \right) a_{2}}{\gamma } &=&B_{1}d_{1},\text{ \ \ \ \ 
}\frac{-2\left( 1+2\lambda \right) a_{3}+\left( 3+6\lambda -\lambda
^{2}\right) a_{2}^{2}}{\gamma }=B_{1}d_{2}+B_{2}d_{1}^{2},  \label{p10}
\end{eqnarray}%
so that, on account of (\ref{p13}) and (\ref{p10}) 
\begin{eqnarray}
c_{1} &=&-d_{1},\text{ \ }  \label{p14} \\
\text{\ }a_{2} &=&\frac{\gamma B_{1}}{\left( 1+\lambda \right) }c_{1}
\label{p100}
\end{eqnarray}%
and%
\begin{equation}
a_{3}=a_{2}^{2}+\frac{\gamma }{4\left( 1+2\lambda \right) }\left[
B_{1}c_{2}+B_{2}c_{1}^{2}-B_{1}d_{2}-B_{2}d_{1}^{2}\right] ,\text{ \ \ \ \ \ 
}\left\vert c_{k}\right\vert \leq 1.  \label{p15}
\end{equation}

Taking into account (\ref{p14}), (\ref{p100}), (\ref{p15}) and Lemma \ref{l1}%
, we obtain

\begin{equation}
\left\vert a_{2}\right\vert =\left\vert \frac{\gamma B_{1}}{\left( 1+\lambda
\right) }c_{1}\right\vert \leq \frac{\left\vert \gamma \right\vert B_{1}}{%
\left( 1+\lambda \right) }  \label{p16}
\end{equation}%
and 
\begin{eqnarray*}
\left\vert a_{3}\right\vert &=&\left\vert a_{2}^{2}+\frac{\gamma }{4\left(
1+2\lambda \right) }\left[
B_{1}c_{2}+B_{2}c_{1}^{2}-B_{1}d_{2}-B_{2}d_{1}^{2}\right] \right\vert \\
&=&\left\vert \frac{\gamma ^{2}B_{1}^{2}}{\left( 1+\lambda \right) ^{2}}%
c_{1}^{2}+\frac{\gamma }{4\left( 1+2\lambda \right) }\left[ \left(
B_{1}c_{2}-B_{2}d_{1}^{2}\right) -\left( B_{1}d_{2}-B_{2}c_{1}^{2}\right) %
\right] \right\vert \\
&=&\left\vert \frac{\gamma ^{2}B_{1}^{2}}{\left( 1+\lambda \right) ^{2}}%
c_{1}^{2}+\frac{\gamma }{4\left( 1+2\lambda \right) }\left[ \left(
B_{1}c_{2}-B_{2}c_{1}^{2}\right) -\left( B_{1}d_{2}-B_{2}d_{1}^{2}\right) %
\right] \right\vert \\
&=&\left\vert \frac{\gamma B_{1}}{4\left( 1+2\lambda \right) }\left\{ \left[
c_{2}-\left( \frac{B_{2}}{B_{1}}-\frac{4\gamma B_{1}\left( 1+2\lambda
\right) }{\left( 1+\lambda \right) ^{2}}\right) c_{1}^{2}\right] -\left[
d_{2}-\frac{B_{2}}{B_{1}}d_{1}^{2}\right] \right\} \right\vert \\
&\leq &\frac{\left\vert \gamma \right\vert B_{1}}{4\left( 1+2\lambda \right) 
}\left\{ \left\vert c_{2}-\left( \frac{B_{2}}{B_{1}}-\frac{4\gamma
B_{1}\left( 1+2\lambda \right) }{\left( 1+\lambda \right) ^{2}}\right)
c_{1}^{2}\right\vert +\left\vert d_{2}-\frac{B_{2}}{B_{1}}%
d_{1}^{2}\right\vert \right\} \\
&\leq &\frac{\left\vert \gamma \right\vert B_{1}}{4\left( 1+2\lambda \right) 
}\left\{ 1+\left( \left\vert s\right\vert -1\right) \left\vert
c_{1}^{2}\right\vert +1+\left( \left\vert t\right\vert -1\right) \left\vert
c_{1}^{2}\right\vert \right\} \\
&=&\frac{\left\vert \gamma \right\vert B_{1}}{4\left( 1+2\lambda \right) }%
\left\{ 2+\left( \left\vert s\right\vert +\left\vert t\right\vert -2\right)
\left\vert c_{1}^{2}\right\vert \right\} \\
&\leq &\frac{\left\vert \gamma \right\vert B_{1}}{4\left( 1+2\lambda \right) 
}\max \left\{ 2,\left( \left\vert s\right\vert +\left\vert t\right\vert
\right) \right\} .
\end{eqnarray*}%
Thus, we have 
\begin{equation*}
\left\vert a_{3}\right\vert \leq \frac{\left\vert \gamma \right\vert
\left\vert B_{1}\right\vert }{4\left( 1+2\lambda \right) }\max \{2,\left(
\left\vert s\right\vert +\left\vert t\right\vert \right) \},
\end{equation*}%
where $s=\frac{B_{2}}{B_{1}}-\frac{4B_{1}\gamma \left( 1+2\lambda \right) }{%
\left( 1+\lambda \right) ^{2}}$ and $t=\frac{B_{2}}{B_{1}}.$

Furthermore, 
\begin{eqnarray}
\left\vert a_{3}-\mu a_{2}^{2}\right\vert &=&\left\vert \left( 1-\mu \right)
a_{2}^{2}+\frac{\gamma }{4\left( 1+2\lambda \right) }\left[
B_{1}c_{2}+B_{2}c_{1}^{2}-B_{1}d_{2}-B_{2}d_{1}^{2}\right] \right\vert 
\notag \\
&=&\left\vert \frac{\gamma B_{1}}{4\left( 1+2\lambda \right) }\left\{ \left[
c_{2}-\left( \frac{B_{2}}{B_{1}}-\left( 1-\mu \right) \frac{4\gamma
B_{1}\left( 1+2\lambda \right) }{\left( 1+\lambda \right) ^{2}}\right)
c_{1}^{2}\right] -\left[ d_{2}-\frac{B_{2}}{B_{1}}d_{1}^{2}\right] \right\}
\right\vert  \notag \\
&\leq &\frac{\left\vert \gamma \right\vert B_{1}}{4\left( 1+2\lambda \right) 
}\left\{ 2+\left( \left\vert \frac{B_{2}}{B_{1}}-\left( 1-\mu \right) \frac{%
4\gamma B_{1}\left( 1+2\lambda \right) }{\left( 1+\lambda \right) ^{2}}%
\right\vert +\left\vert \frac{B_{2}}{B_{1}}\right\vert -2\right) \left\vert
c_{1}^{2}\right\vert \right\} .  \label{p17}
\end{eqnarray}%
As a result of this, we obtain%
\begin{equation*}
\left\vert a_{3}-\mu a_{2}^{2}\right\vert \leq \left\{ 
\begin{array}{c}
\frac{B_{1}\left\vert \gamma \right\vert }{2\left( 1+2\lambda \right) }\text{%
\ \ \ \ \ \ \ \ \ \ \ \ \ \ \ \ if \ } \mathcal{L} <2, \\ 
\frac{B_{1}\left\vert \gamma \right\vert }{4\left( 1+2\lambda \right) }%
\mathcal{L} \text{\ \ \ \ \ \ \ \ \ \ \ \ \ if \ }
\mathcal{L} \geq 2,
\end{array}
\right. .\text{ }
\end{equation*}%
where $\mathcal{L} =\left\vert \frac{B_{2}}{B_{1}}+\left( 1-\mu
\right) \frac{4B_{1}\gamma \left( 1+2\lambda \right) }{\left( 1+\lambda
\right) ^{2}}\right\vert +\left\vert \frac{B_{2}}{B_{1}}\right\vert .$

Thus the proof is completed.
\end{proof}

We next consider the case, \ when $\gamma $ and $\mu $ are real. Then we
have:

\begin{theorem}
\label{T2}Let $\gamma >0$ and if both functions $f$ of the form (\ref{1})
and its inverse maps $g=f^{-1}$ are in $\mathcal{S}(\lambda ,\gamma ;\varphi
),$ then for $\mu \in 
\mathbb{R}
,$
\begin{itemize}
\item[(1)] If $\left\vert B_{2}\right\vert \geq B_{1},$ we have 
\begin{equation*}
\left\vert a_{3}-\mu a_{2}^{2}\right\vert \leq \left\{ 
\begin{array}{c}
\frac{\gamma \left\vert B_{2}\right\vert }{2\left( 1+2\lambda \right) }%
-\left( \mu -1\right) \text{ }\frac{\gamma ^{2}B_{1}^{2}}{\left( 1+\lambda
\right) ^{2}}\text{\ \ \ \ \ \ if \ }\mu \leq 1, \\ 
\frac{\gamma \left\vert B_{2}\right\vert }{2\left( 1+2\lambda \right) }%
+\left( \mu -1\right) \text{ }\frac{\gamma ^{2}B_{1}^{2}}{\left( 1+\lambda
\right) ^{2}}\text{\ \ \ \ \ \ if \ }\mu >1,%
\end{array}%
\right.
\end{equation*}%
\item[(2)] If $\left\vert B_{2}\right\vert <B_{1},$ we have 
\begin{equation*}
\left\vert a_{3}-\mu a_{2}^{2}\right\vert \leq \left\{ 
\begin{array}{c}
\frac{\gamma \left\vert B_{2}\right\vert }{2\left( 1+2\lambda \right) }%
-\left( \mu -1\right) \text{ }\frac{\gamma ^{2}B_{1}^{2}}{\left( 1+\lambda
\right) ^{2}}\text{\ \ \ \ \ \ \ \ \ \ \ \ \ \ \ \ if \ }\mu \leq 1-%
{\normalsize \mathcal{F} }, \\ 
\frac{\gamma B_{1}}{2\left( 1+2\lambda \right) }\text{\ \ \ \ \ \ \ \ \
\ \ \ \ \ \ \ \ \ \ \ \ \ \ \ \ \ \ \ \ \ if }1-{\normalsize \mathcal{F} <}%
\mu <1+{\normalsize \mathcal{F} }\text{\ }, \\ 
\frac{\gamma \left\vert B_{2}\right\vert }{2\left( 1+2\lambda \right) }%
+\left( \mu -1\right) \text{ }\frac{\gamma ^{2}B_{1}^{2}}{\left( 1+\lambda
\right) ^{2}}\text{\ \ \ \ \ \ \ \ \ \ \ \ \ \ \ if \ }\mu \geq 1+%
{\normalsize \mathcal{F} },%
\end{array}%
\right.
\end{equation*}
\end{itemize}
where ${\normalsize \mathcal{F} =}\frac{\left( 1+\lambda \right) ^{2}\left(
B_{1}-\left\vert B_{2}\right\vert \right) }{2\gamma B_{1}^{2}\left(
1+2\lambda \right) }.$

For each $\mu $ there is a function $f\in \mathcal{S}(\lambda ,\gamma
;\varphi )$ such that equality holds.

\begin{proof}
Using (\ref{p17}) and Lemma \ref{l1}, we obtain\ 
\begin{eqnarray}
\left\vert a_{3}-\mu a_{2}^{2}\right\vert &=&\left\vert \frac{\gamma B_{1}}{%
4\left( 1+2\lambda \right) }\left\{ \left[ c_{2}-\left( \frac{B_{2}}{B_{1}}%
-\left( 1-\mu \right) \frac{4\gamma B_{1}\left( 1+2\lambda \right) }{\left(
1+\lambda \right) ^{2}}\right) c_{1}^{2}\right] -\left[ d_{2}-\frac{B_{2}}{%
B_{1}}d_{1}^{2}\right] \right\} \right\vert  \notag \\
&\leq &\frac{\left\vert \gamma \right\vert B_{1}}{4\left( 1+2\lambda \right) 
}\left\{ 2+\left( \left\vert \frac{B_{2}}{B_{1}}-\left( 1-\mu \right) \frac{%
4\gamma B_{1}\left( 1+2\lambda \right) }{\left( 1+\lambda \right) ^{2}}%
\right\vert +\left\vert \frac{B_{2}}{B_{1}}\right\vert -2\right) \left\vert
c_{1}^{2}\right\vert \right\}  \notag \\
&=&\frac{\gamma B_{1}}{2\left( 1+2\lambda \right) }+\left\{ \frac{\gamma
\left( \left\vert B_{2}\right\vert -B_{1}\right) }{2\left( 1+2\lambda
\right) }+\left\vert \mu -1\right\vert \frac{\gamma ^{2}B_{1}^{2}}{\left(
1+\lambda \right) ^{2}}\right\} \left\vert c_{1}^{2}\right\vert .
\label{p21}
\end{eqnarray}%
Now, the proof will be presented in two cases by considering $B_{1}$, $%
B_{2}\in 
\mathbb{R}
$ and$\ B_{1}>0.$

Firstly, we want to consider the case with $\left\vert B_{2}\right\vert \geq
B_{1}.$ Several possible cases need to further analyze.
\begin{itemize}
\item[Case 1]:  If $\mu \leq 1$, using (\ref{p21}) and Lemma \ref{l1}, we obtain\ 
\begin{eqnarray*}
\left\vert a_{3}-\mu a_{2}^{2}\right\vert &\leq &\frac{\gamma B_{1}}{2\left(
1+2\lambda \right) }+\left\{ \frac{\gamma \left( \left\vert B_{2}\right\vert
-B_{1}\right) }{2\left( 1+2\lambda \right) }+\left( 1-\mu \right) \frac{%
\gamma ^{2}B_{1}^{2}}{\left( 1+\lambda \right) ^{2}}\right\} \left\vert
c_{1}^{2}\right\vert \\
&\leq &\frac{\gamma B_{1}}{2\left( 1+2\lambda \right) }+\left\{ \frac{\gamma
\left( \left\vert B_{2}\right\vert -B_{1}\right) }{2\left( 1+2\lambda
\right) }+\left( 1-\mu \right) \frac{\gamma ^{2}B_{1}^{2}}{\left( 1+\lambda
\right) ^{2}}\right\} \\
&=&\frac{\gamma B_{1}}{2\left( 1+2\lambda \right) }+\frac{\gamma \left\vert
B_{2}\right\vert }{2\left( 1+2\lambda \right) }-\frac{\gamma B_{1}}{2\left(
1+2\lambda \right) }+\frac{\gamma ^{2}B_{1}^{2}}{\left( 1+\lambda \right)
^{2}}-\mu \frac{\gamma ^{2}B_{1}^{2}}{\left( 1+\lambda \right) ^{2}} \\
&=&\frac{\gamma \left\vert B_{2}\right\vert }{2\left( 1+2\lambda \right) }%
-\left( \mu -1\right) \frac{\gamma ^{2}B_{1}^{2}}{\left( 1+\lambda \right)
^{2}}.
\end{eqnarray*}%
\item[Case 2]: \ If $\mu >1$, again using (\ref{p21}) and Lemma \ref{l1}, we obtain%
\begin{eqnarray*}
\left\vert a_{3}-\mu a_{2}^{2}\right\vert &\leq &\frac{\gamma B_{1}}{2\left(
1+2\lambda \right) }+\left\{ \frac{\gamma \left( \left\vert B_{2}\right\vert
-B_{1}\right) }{2\left( 1+2\lambda \right) }+\left( \mu -1\right) \frac{%
\gamma ^{2}B_{1}^{2}}{\left( 1+\lambda \right) ^{2}}\right\} \left\vert
c_{1}^{2}\right\vert \\
&\leq &\frac{\gamma B_{1}}{2\left( 1+2\lambda \right) }+\left\{ \frac{\gamma
\left( \left\vert B_{2}\right\vert -B_{1}\right) }{2\left( 1+2\lambda
\right) }+\left( \mu -1\right) \frac{\gamma ^{2}B_{1}^{2}}{\left( 1+\lambda
\right) ^{2}}\right\} \\
&=&\frac{\gamma B_{1}}{2\left( 1+2\lambda \right) }+\frac{\gamma \left\vert
B_{2}\right\vert }{2\left( 1+2\lambda \right) }-\frac{\gamma B_{1}}{2\left(
1+2\lambda \right) }-\frac{\gamma ^{2}B_{1}^{2}}{\left( 1+\lambda \right)
^{2}}+\mu \frac{\gamma ^{2}B_{1}^{2}}{\left( 1+\lambda \right) ^{2}} \\
&=&\frac{\gamma \left\vert B_{2}\right\vert }{2\left( 1+2\lambda \right) }%
+\left( \mu -1\right) \frac{\gamma ^{2}B_{1}^{2}}{\left( 1+\lambda \right)
^{2}}\cdot
\end{eqnarray*}\end{itemize}
Finally, we want to consider the case with $\left\vert B_{2}\right\vert
<B_{1}.$ By a similar way, several possible cases need to further analyze.
\begin{itemize}
\item[(i)] \ Let $\mu \leq 1-{\normalsize \mathcal{F} },$ using (\ref{p21}) and
Lemma \ref{l1}, we have%
\begin{eqnarray*}
\left\vert a_{3}-\mu a_{2}^{2}\right\vert &\leq &\frac{\gamma B_{1}}{2\left(
1+2\lambda \right) }+\left\{ \frac{\gamma \left( \left\vert B_{2}\right\vert
-B_{1}\right) }{2\left( 1+2\lambda \right) }+\left( 1-\mu \right) \frac{%
\gamma ^{2}B_{1}^{2}}{\left( 1+\lambda \right) ^{2}}\right\} \left\vert
c_{1}^{2}\right\vert \\
&\leq &\frac{\gamma B_{1}}{2\left( 1+2\lambda \right) }+\left\{ \frac{\gamma
\left( \left\vert B_{2}\right\vert -B_{1}\right) }{2\left( 1+2\lambda
\right) }+\left( 1-\mu \right) \frac{\gamma ^{2}B_{1}^{2}}{\left( 1+\lambda
\right) ^{2}}\right\} \\
&=&\frac{\gamma \left\vert B_{2}\right\vert }{2\left( 1+2\lambda \right) }%
-\left( \mu -1\right) \frac{\gamma ^{2}B_{1}^{2}}{\left( 1+\lambda \right)
^{2}}.
\end{eqnarray*}%
\item[(ii)] \ Let $1-{\normalsize \mathcal{F} <}\mu \leq 1,$ using (\ref{p21}) and
Lemma \ref{l1}, we obtain%
\begin{eqnarray*}
\left\vert a_{3}-\mu a_{2}^{2}\right\vert &\leq &\frac{\gamma B_{1}}{2\left(
1+2\lambda \right) }+\left\{ \frac{\gamma \left( \left\vert B_{2}\right\vert
-B_{1}\right) }{2\left( 1+2\lambda \right) }+\left( 1-\mu \right) \frac{%
\gamma ^{2}B_{1}^{2}}{\left( 1+\lambda \right) ^{2}}\right\} \left\vert
c_{1}^{2}\right\vert \\
&\leq &\frac{\gamma B_{1}}{2\left( 1+2\lambda \right) }.
\end{eqnarray*}%
\item[(iii)] \ Let $1<\mu <1+{\normalsize \mathcal{F} },$ using (\ref{p21}) and
Lemma \ref{l1}, we obtain%
\begin{eqnarray*}
\left\vert a_{3}-\mu a_{2}^{2}\right\vert &\leq &\frac{\gamma B_{1}}{2\left(
1+2\lambda \right) }+\left\{ \frac{\gamma \left( \left\vert B_{2}\right\vert
-B_{1}\right) }{2\left( 1+2\lambda \right) }+\left( \mu -1\right) \frac{%
\gamma ^{2}B_{1}^{2}}{\left( 1+\lambda \right) ^{2}}\right\} \left\vert
c_{1}^{2}\right\vert \\
&\leq &\frac{\gamma B_{1}}{2\left( 1+2\lambda \right) }.
\end{eqnarray*}%
\item[(iv)] \ Let $\mu \geq 1+{\normalsize \mathcal{F} },$ using (\ref{p21}) and
Lemma \ref{l1}, we have%
\begin{eqnarray*}
\left\vert a_{3}-\mu a_{2}^{2}\right\vert &\leq &\frac{\gamma B_{1}}{2\left(
1+2\lambda \right) }+\left\{ \frac{\gamma \left( \left\vert B_{2}\right\vert
-B_{1}\right) }{2\left( 1+2\lambda \right) }+\left( \mu -1\right) \frac{%
\gamma ^{2}B_{1}^{2}}{\left( 1+\lambda \right) ^{2}}\right\} \left\vert
c_{1}^{2}\right\vert \\
&\leq &\frac{\gamma B_{1}}{2\left( 1+2\lambda \right) }+\left\{ \frac{\gamma
\left( \left\vert B_{2}\right\vert -B_{1}\right) }{2\left( 1+2\lambda
\right) }+\left( \mu -1\right) \frac{\gamma ^{2}B_{1}^{2}}{\left( 1+\lambda
\right) ^{2}}\right\} \\
&=&\frac{\gamma B_{1}}{2\left( 1+2\lambda \right) }+\frac{\gamma \left\vert
B_{2}\right\vert }{2\left( 1+2\lambda \right) }-\frac{\gamma B_{1}}{2\left(
1+2\lambda \right) }-\frac{\gamma ^{2}B_{1}^{2}}{\left( 1+\lambda \right)
^{2}}+\mu \frac{\gamma ^{2}B_{1}^{2}}{\left( 1+\lambda \right) ^{2}} \\
&=&\frac{\gamma \left\vert B_{2}\right\vert }{2\left( 1+2\lambda \right) }%
+\left( \mu -1\right) \frac{\gamma ^{2}B_{1}^{2}}{\left( 1+\lambda \right)
^{2}}.
\end{eqnarray*}
\end{itemize}
Thus the proof is completed.
\end{proof}

Finally, we consider the case, when $\gamma $ nonzero complex number and $%
\mu \in 
\mathbb{C}
.$ Then we get:
\end{theorem}

\begin{theorem}
\label{T3} \textit{Let }$\gamma $ be a nonzero complex number and let both
functions $f$ of the form (\ref{1}) and its inverse maps $g=f^{-1}$ are in $%
\mathcal{S}(\lambda ,\gamma ;\varphi ).$ Then for $\mu \in 
\mathbb{R}
$ we have \textit{\ }
\begin{itemize}

\item[(1)] \ If $\ \frac{\left( 1+\left\vert \sin \theta \right\vert \right)
\left\vert B_{2}\right\vert }{2B_{1}}\geq 1,$ we have 
\begin{equation*}
\left\vert a_{3}-\mu a_{2}^{2}\right\vert \leq \left\{ 
\begin{array}{c}
\frac{\left\vert \gamma \right\vert ^{2}B_{1}^{2}}{\left( 1+\lambda \right)
^{2}}\left( 1-\mu -\Re \left( k_{1}\right) \right) +\frac{\left\vert \gamma
\right\vert \left\vert B_{2}\right\vert \left( 1+\left\vert \sin \theta
\right\vert \right) }{4\left( 1+2\lambda \right) }\text{\ \ \ \ \ if \ }\mu
\leq 1-\Re \left( k_{1}\right) , \\ 
\frac{\left\vert \gamma \right\vert \left\vert B_{2}\right\vert \left(
1+\left\vert \sin \theta \right\vert \right) }{4\left( 1+2\lambda \right) }-%
\frac{\left\vert \gamma \right\vert ^{2}B_{1}^{2}}{\left( 1+\lambda \right)
^{2}}\left( 1-\mu -\Re \left( k_{1}\right) \right) \text{\ \ \ \ \ if \ }\mu
>1-\Re \left( k_{1}\right) .%
\end{array}%
\right.
\end{equation*}%
\item[(2)] \ If $\ \frac{\left( 1+\left\vert \sin \theta \right\vert \right)
\left\vert B_{2}\right\vert }{2B_{1}}<1,$ we obtain 
\begin{equation*}
\left\vert a_{3}-\mu a_{2}^{2}\right\vert \leq \left\{ 
\begin{array}{c}
\frac{\left\vert \gamma \right\vert ^{2}B_{1}^{2}}{\left( 1+\lambda \right)
^{2}}\left( 1-\mu -\Re \left( k_{1}\right) \right) +\frac{\left\vert \gamma
\right\vert \left\vert B_{2}\right\vert \left( 1+\left\vert \sin \theta
\right\vert \right) }{4\left( 1+2\lambda \right) }\text{\ \ \ \ \ \ \ \ \ \
\ \ \ \ \ \ \ \ \ \ \ \ \ \ \ \ \ \ \ \ if \ }\mu \leq 1-\Re \left(
k_{1}\right) +\mathcal{N}, \\ 
\frac{\left\vert \gamma \right\vert B_{1}}{2\left( 1+2\lambda \right) }\text{%
\ \ \ \ \ \ \ \ \ \ \ \ \ \ \ \ \ \ \ \ \ \ \ \ \ \ \ \ \ \ \ \ \ \ \ \ \ \
\ \ \ \ \ \ \ \ \ \ \ \ \ if \ }1-\Re \left( k_{1}\right) +\mathcal{N}%
{\normalsize <}\mu <1-\Re \left( k_{1}\right) -\mathcal{N}\text{\ }, \\ 
\frac{\left\vert \gamma \right\vert \left\vert B_{2}\right\vert \left(
1+\left\vert \sin \theta \right\vert \right) }{4\left( 1+2\lambda \right) }-%
\frac{\left\vert \gamma \right\vert ^{2}B_{1}^{2}}{\left( 1+\lambda \right)
^{2}}\left( 1-\mu -\Re \left( k_{1}\right) \right) \text{\ \ \ \ \ \ \ \ \ \
\ \ \ \ \ \ \ \ \ \ \ \ \ \ \ \ \ \ \ \ if \ }\mu \geq 1-\Re \left(
k_{1}\right) -\mathcal{N},%
\end{array}%
\right.
\end{equation*}%
where $k_{1}=\frac{B_{2}\left( 1+\lambda \right) ^{2}e^{i\theta }}{%
4B_{1}^{2}\left\vert \gamma \right\vert \left( 1+2\lambda \right) },$ $l_{1=}%
\frac{\left( \left\vert B_{2}\right\vert -2B_{1}\right) \left( 1+\lambda
\right) ^{2}}{4B_{1}^{2}\left\vert \gamma \right\vert \left( 1+2\lambda
\right) },$ $\left\vert \gamma \right\vert =\gamma e^{i\theta }\ $and $%
\mathcal{N=}\frac{\left( 1+\lambda \right) ^{2}\left[ \left\vert
B_{2}\right\vert \left( 1+\left\vert \sin \theta \right\vert \right) -2B_{1}%
\right] }{4B_{1}^{2}\left\vert \gamma \right\vert \left( 1+2\lambda \right) }%
.$
\end{itemize}
For each $\mu $ there is a function in $\mathcal{S}(\lambda ,\gamma ;\varphi
)$ such that the equality holds.
\end{theorem}

\begin{proof}
\noindent Suppose $f\left( z\right) =z+\sum_{k=2}^{\infty }a_{k}z^{k}\in 
\mathcal{S}(\lambda ,\gamma ;\varphi ),$ using (\ref{p17}) and Lemma \ref{l1}%
, then we obtain%
\begin{eqnarray}
\left\vert a_{3}-\mu a_{2}^{2}\right\vert &\leq &\frac{\left\vert \gamma
\right\vert B_{1}}{4\left( 1+2\lambda \right) }\left\{ 2+\left( \left\vert 
\frac{B_{2}}{B_{1}}-\left( 1-\mu \right) \frac{4\gamma B_{1}\left(
1+2\lambda \right) }{\left( 1+\lambda \right) ^{2}}\right\vert +\left\vert 
\frac{B_{2}}{B_{1}}\right\vert -2\right) \left\vert c_{1}^{2}\right\vert
\right\}  \notag \\
&=&\frac{\left\vert \gamma \right\vert B_{1}}{2\left( 1+2\lambda \right) }+%
\frac{\left\vert \gamma \right\vert ^{2}B_{1}^{2}}{\left( 1+\lambda \right)
^{2}}\left[ \left\vert \left( 1-\mu \right) -\frac{B_{2}\left( 1+\lambda
\right) ^{2}}{4B_{1}^{2}\gamma \left( 1+2\lambda \right) }\right\vert +\frac{%
\left( \left\vert B_{2}\right\vert -2B_{1}\right) \left( 1+\lambda \right)
^{2}}{4B_{1}^{2}\left\vert \gamma \right\vert \left( 1+2\lambda \right) }%
\right] \left\vert c_{1}^{2}\right\vert .  \label{p31}
\end{eqnarray}%
Taking $\left\vert \gamma \right\vert =\gamma e^{i\theta }$, $k_{1}=\frac{%
B_{2}\left( 1+\lambda \right) ^{2}e^{i\theta }}{4B_{1}^{2}\left\vert \gamma
\right\vert \left( 1+2\lambda \right) }$ and $l_{1}=\frac{\left( \left\vert
B_{2}\right\vert -2B_{1}\right) \left( 1+\lambda \right) ^{2}}{%
4B_{1}^{2}\left\vert \gamma \right\vert \left( 1+2\lambda \right) }$, a
direct calculation with (\ref{p31}) shows that 
\begin{eqnarray}
\left\vert a_{3}-\mu a_{2}^{2}\right\vert &\leq &\frac{\left\vert \gamma
\right\vert B_{1}}{2\left( 1+2\lambda \right) }+\frac{\left\vert \gamma
\right\vert ^{2}B_{1}^{2}}{\left( 1+\lambda \right) ^{2}}\left( \left\vert
1-\mu -k_{1}\right\vert +l_{1}\right) \left\vert c_{1}^{2}\right\vert  \notag
\\
&\leq &\frac{\left\vert \gamma \right\vert B_{1}}{2\left( 1+2\lambda \right) 
}+\frac{\left\vert \gamma \right\vert ^{2}B_{1}^{2}}{\left( 1+\lambda
\right) ^{2}}\left( \left\vert 1-\mu -\Re \left( k_{1}\right) -i\left( \mathbf{%
Im}(k_{1}\right) \right\vert +l_{1}\right) \left\vert c_{1}^{2}\right\vert 
\notag \\
&\leq &\frac{\left\vert \gamma \right\vert B_{1}}{2\left( 1+2\lambda \right) 
}+\frac{\left\vert \gamma \right\vert ^{2}B_{1}^{2}}{\left( 1+\lambda
\right) ^{2}}\left( \left\vert 1-\mu -\Re \left( k_{1}\right) \right\vert +%
\frac{\left\vert B_{2}\right\vert \left( 1+\lambda \right) ^{2}\left\vert
\sin \theta \right\vert }{4B_{1}^{2}\left\vert \gamma \right\vert \left(
1+2\lambda \right) }+l_{1}\right) \left\vert c_{1}^{2}\right\vert  \notag \\
&=&\frac{\left\vert \gamma \right\vert B_{1}}{2\left( 1+2\lambda \right) }+%
\frac{\left\vert \gamma \right\vert ^{2}B_{1}^{2}}{\left( 1+\lambda \right)
^{2}}\left( \left\vert 1-\mu -\Re \left( k_{1}\right) \right\vert +\frac{%
\left\vert B_{2}\right\vert \left\vert \sin \theta \right\vert }{\left\vert
B_{2}\right\vert -2B_{1}}l_{1}+l_{1}\right) \left\vert c_{1}^{2}\right\vert 
\notag \\
&=&\frac{\left\vert \gamma \right\vert B_{1}}{2\left( 1+2\lambda \right) }+%
\frac{\left\vert \gamma \right\vert ^{2}B_{1}^{2}}{\left( 1+\lambda \right)
^{2}}\left[ \left\vert 1-\mu -\Re \left( k_{1}\right) \right\vert
+l_{1}\left( \frac{\left\vert B_{2}\right\vert \left\vert \sin \theta
\right\vert }{\left\vert B_{2}\right\vert -2B_{1}}+1\right) \right]
\left\vert c_{1}^{2}\right\vert  \notag \\
&=&\frac{\left\vert \gamma \right\vert B_{1}}{2\left( 1+2\lambda \right) }+%
\left[ \frac{\left\vert \gamma \right\vert ^{2}B_{1}^{2}}{\left( 1+\lambda
\right) ^{2}}\cdot \left\vert 1-\mu -\Re \left( k_{1}\right) \right\vert +%
\frac{\left\vert \gamma \right\vert \left[ \left\vert B_{2}\right\vert
\left( 1+\left\vert \sin \theta \right\vert \right) -2B_{1}\right] }{4\left(
1+2\lambda \right) }\right] \left\vert c_{1}^{2}\right\vert .  \label{p32}
\end{eqnarray}%
Now, we will make some discussions for several different cases by
considering $B_{1}$, $B_{2}\in 
\mathbb{R}
$ and$\ B_{1}>0.$

Firstly, we want to consider the case with $\frac{2B_{1}}{\left\vert
B_{2}\right\vert }-\left\vert \sin \theta \right\vert <1.$ Several possible
cases need to further analyze.
\begin{itemize}

\item[Case 1]: Let $\mu \leq 1-\Re \left( k_{1}\right) $. Then from (\ref{p32}) and
Lemma \ref{l1}, we obtain%
\begin{eqnarray*}
\left\vert a_{3}-\mu a_{2}^{2}\right\vert &\leq &\frac{\left\vert \gamma
\right\vert B_{1}}{2\left( 1+2\lambda \right) }+\left[ \frac{\left\vert
\gamma \right\vert ^{2}B_{1}^{2}}{\left( 1+\lambda \right) ^{2}}\cdot
\left\vert 1-\mu -\Re \left( k_{1}\right) \right\vert +\frac{\left\vert
\gamma \right\vert \left[ \left\vert B_{2}\right\vert \left( 1+\left\vert
\sin \theta \right\vert \right) -2B_{1}\right] }{4\left( 1+2\lambda \right) }%
\right] \left\vert c_{1}^{2}\right\vert \\
&\leq &\frac{\left\vert \gamma \right\vert B_{1}}{2\left( 1+2\lambda \right) 
}+\frac{\left\vert \gamma \right\vert ^{2}B_{1}^{2}}{\left( 1+\lambda
\right) ^{2}}\cdot \left( 1-\mu -\Re \left( k_{1}\right) \right) +\frac{%
\left\vert \gamma \right\vert \left[ \left\vert B_{2}\right\vert \left(
1+\left\vert \sin \theta \right\vert \right) -2B_{1}\right] }{4\left(
1+2\lambda \right) } \\
&=&\frac{\left\vert \gamma \right\vert B_{1}}{2\left( 1+2\lambda \right) }+%
\frac{\left\vert \gamma \right\vert ^{2}B_{1}^{2}}{\left( 1+\lambda \right)
^{2}}\cdot \left( 1-\mu -\Re \left( k_{1}\right) \right) +\frac{\left\vert
\gamma \right\vert \left\vert B_{2}\right\vert \left( 1+\left\vert \sin
\theta \right\vert \right) }{4\left( 1+2\lambda \right) }-\frac{\left\vert
\gamma \right\vert B_{1}}{2\left( 1+2\lambda \right) } \\
&=&\frac{\left\vert \gamma \right\vert ^{2}B_{1}^{2}}{\left( 1+\lambda
\right) ^{2}}\cdot \left( 1-\mu -\Re \left( k_{1}\right) \right) +\frac{%
\left\vert \gamma \right\vert \left\vert B_{2}\right\vert \left(
1+\left\vert \sin \theta \right\vert \right) }{4\left( 1+2\lambda \right) }.
\end{eqnarray*}%
\item[Case 2]: Let $\mu >1-\Re \left( k_{1}\right) $, then from (\ref{p32}) and
Lemma \ref{l1}, we yield%
\begin{eqnarray*}
\left\vert a_{3}-\mu a_{2}^{2}\right\vert &\leq &\frac{\left\vert \gamma
\right\vert B_{1}}{2\left( 1+2\lambda \right) }+\left[ \frac{\left\vert
\gamma \right\vert ^{2}B_{1}^{2}}{\left( 1+\lambda \right) ^{2}}\cdot
\left\vert 1-\mu -\Re \left( k_{1}\right) \right\vert +\frac{\left\vert
\gamma \right\vert \left[ \left\vert B_{2}\right\vert \left( 1+\left\vert
\sin \theta \right\vert \right) -2B_{1}\right] }{4\left( 1+2\lambda \right) }%
\right] \left\vert c_{1}^{2}\right\vert \\
&\leq &\frac{\left\vert \gamma \right\vert B_{1}}{2\left( 1+2\lambda \right) 
}+\frac{\left\vert \gamma \right\vert ^{2}B_{1}^{2}}{\left( 1+\lambda
\right) ^{2}}\cdot \left( \mu +\Re \left( k_{1}\right) -1\right) +\frac{%
\left\vert \gamma \right\vert \left[ \left\vert B_{2}\right\vert \left(
1+\left\vert \sin \theta \right\vert \right) -2B_{1}\right] }{4\left(
1+2\lambda \right) } \\
&=&\frac{\left\vert \gamma \right\vert B_{1}}{2\left( 1+2\lambda \right) }-%
\frac{\left\vert \gamma \right\vert ^{2}B_{1}^{2}}{\left( 1+\lambda \right)
^{2}}\cdot \left( 1-\mu -\Re \left( k_{1}\right) \right) +\frac{\left\vert
\gamma \right\vert \left\vert B_{2}\right\vert \left( 1+\left\vert \sin
\theta \right\vert \right) }{4\left( 1+2\lambda \right) }-\frac{\left\vert
\gamma \right\vert B_{1}}{2\left( 1+2\lambda \right) } \\
&=&\frac{\left\vert \gamma \right\vert \left\vert B_{2}\right\vert \left(
1+\left\vert \sin \theta \right\vert \right) }{4\left( 1+2\lambda \right) }-%
\frac{\left\vert \gamma \right\vert ^{2}B_{1}^{2}}{\left( 1+\lambda \right)
^{2}}\cdot \left( 1-\mu -\Re \left( k_{1}\right) \right) .
\end{eqnarray*}\end{itemize}
Finally, we want to consider the case with $\frac{2B_{1}}{\left\vert
B_{2}\right\vert }-\left\vert \sin \theta \right\vert >1.$ By a similar
approximation, several possible cases need to further analyze.
\begin{itemize}

\item[(i)] \ Let \ $\mu \leq 1-\Re \left( k_{1}\right) +\mathcal{N},$ using (\ref%
{p32}) and Lemma \ref{l1}, we have%
\begin{eqnarray*}
\left\vert a_{3}-\mu a_{2}^{2}\right\vert &\leq &\frac{\left\vert \gamma
\right\vert B_{1}}{2\left( 1+2\lambda \right) }+\left[ \frac{\left\vert
\gamma \right\vert ^{2}B_{1}^{2}}{\left( 1+\lambda \right) ^{2}}\cdot
\left\vert 1-\mu -\Re \left( k_{1}\right) \right\vert +\frac{\left\vert
\gamma \right\vert \left[ \left\vert B_{2}\right\vert \left( 1+\left\vert
\sin \theta \right\vert \right) -2B_{1}\right] }{4\left( 1+2\lambda \right) }%
\right] \left\vert c_{1}^{2}\right\vert \\
&\leq &\frac{\left\vert \gamma \right\vert B_{1}}{2\left( 1+2\lambda \right) 
}+\frac{\left\vert \gamma \right\vert ^{2}B_{1}^{2}}{\left( 1+\lambda
\right) ^{2}}\cdot \left( 1-\mu -\Re \left( k_{1}\right) \right) +\frac{%
\left\vert \gamma \right\vert \left[ \left\vert B_{2}\right\vert \left(
1+\left\vert \sin \theta \right\vert \right) -2B_{1}\right] }{4\left(
1+2\lambda \right) } \\
&=&\frac{\left\vert \gamma \right\vert B_{1}}{2\left( 1+2\lambda \right) }+%
\frac{\left\vert \gamma \right\vert ^{2}B_{1}^{2}}{\left( 1+\lambda \right)
^{2}}\cdot \left( 1-\mu -\Re \left( k_{1}\right) \right) +\frac{\left\vert
\gamma \right\vert \left\vert B_{2}\right\vert \left( 1+\left\vert \sin
\theta \right\vert \right) }{4\left( 1+2\lambda \right) }-\frac{\left\vert
\gamma \right\vert B_{1}}{2\left( 1+2\lambda \right) } \\
&=&\frac{\left\vert \gamma \right\vert ^{2}B_{1}^{2}}{\left( 1+\lambda
\right) ^{2}}\cdot \left( 1-\mu -\Re \left( k_{1}\right) \right) +\frac{%
\left\vert \gamma \right\vert \left\vert B_{2}\right\vert \left(
1+\left\vert \sin \theta \right\vert \right) }{4\left( 1+2\lambda \right) }.
\end{eqnarray*}%
\item[(ii)] \ Let $1-\Re \left( k_{1}\right) +\mathcal{N}{\normalsize <}\mu \leq
1-\Re \left( k_{1}\right) ,$ using (\ref{p32}) and Lemma \ref{l1}, we obtain%
\begin{eqnarray*}
\left\vert a_{3}-\mu a_{2}^{2}\right\vert &\leq &\frac{\left\vert \gamma
\right\vert B_{1}}{2\left( 1+2\lambda \right) }+\left[ \frac{\left\vert
\gamma \right\vert ^{2}B_{1}^{2}}{\left( 1+\lambda \right) ^{2}}\cdot
\left\vert 1-\mu -\Re \left( k_{1}\right) \right\vert +\frac{\left\vert
\gamma \right\vert \left[ \left\vert B_{2}\right\vert \left( 1+\left\vert
\sin \theta \right\vert \right) -2B_{1}\right] }{4\left( 1+2\lambda \right) }%
\right] \left\vert c_{1}^{2}\right\vert \\
&\leq &\frac{\left\vert \gamma \right\vert B_{1}}{2\left( 1+2\lambda \right) 
}.
\end{eqnarray*}%
\item[(iii)] \ Let $1-\Re \left( k_{1}\right) <\mu <1-\Re \left( k_{1}\right) -%
\mathcal{N},$ using (\ref{p32}) and Lemma \ref{l1}, we obtain%
\begin{eqnarray*}
\left\vert a_{3}-\mu a_{2}^{2}\right\vert &\leq &\frac{\left\vert \gamma
\right\vert B_{1}}{2\left( 1+2\lambda \right) }+\left[ \frac{\left\vert
\gamma \right\vert ^{2}B_{1}^{2}}{\left( 1+\lambda \right) ^{2}}\cdot
\left\vert 1-\mu -\Re \left( k_{1}\right) \right\vert +\frac{\left\vert
\gamma \right\vert \left[ \left\vert B_{2}\right\vert \left( 1+\left\vert
\sin \theta \right\vert \right) -2B_{1}\right] }{4\left( 1+2\lambda \right) }%
\right] \left\vert c_{1}^{2}\right\vert \\
&\leq &\frac{\left\vert \gamma \right\vert B_{1}}{2\left( 1+2\lambda \right) 
}.
\end{eqnarray*}%
\item[(iv)] \ Let $\mu \geq 1-\Re \left( k_{1}\right) -\mathcal{N},$ using (\ref%
{p32}) and Lemma \ref{l1}, we have%
\begin{eqnarray*}
\left\vert a_{3}-\mu a_{2}^{2}\right\vert &\leq &\frac{\left\vert \gamma
\right\vert B_{1}}{2\left( 1+2\lambda \right) }+\left[ \frac{\left\vert
\gamma \right\vert ^{2}B_{1}^{2}}{\left( 1+\lambda \right) ^{2}}\cdot
\left\vert 1-\mu -\Re \left( k_{1}\right) \right\vert +\frac{\left\vert
\gamma \right\vert \left[ \left\vert B_{2}\right\vert \left( 1+\left\vert
\sin \theta \right\vert \right) -2B_{1}\right] }{4\left( 1+2\lambda \right) }%
\right] \left\vert c_{1}^{2}\right\vert \\
&\leq &\frac{\left\vert \gamma \right\vert B_{1}}{2\left( 1+2\lambda \right) 
}+\frac{\left\vert \gamma \right\vert ^{2}B_{1}^{2}}{\left( 1+\lambda
\right) ^{2}}\cdot \left( \mu +\Re \left( k_{1}\right) -1\right) +\frac{%
\left\vert \gamma \right\vert \left[ \left\vert B_{2}\right\vert \left(
1+\left\vert \sin \theta \right\vert \right) -2B_{1}\right] }{4\left(
1+2\lambda \right) } \\
&=&\frac{\left\vert \gamma \right\vert B_{1}}{2\left( 1+2\lambda \right) }-%
\frac{\left\vert \gamma \right\vert ^{2}B_{1}^{2}}{\left( 1+\lambda \right)
^{2}}\cdot \left( 1-\mu -\Re \left( k_{1}\right) \right) +\frac{\left\vert
\gamma \right\vert \left\vert B_{2}\right\vert \left( 1+\left\vert \sin
\theta \right\vert \right) }{4\left( 1+2\lambda \right) }-\frac{\left\vert
\gamma \right\vert B_{1}}{2\left( 1+2\lambda \right) } \\
&=&\frac{\left\vert \gamma \right\vert \left\vert B_{2}\right\vert \left(
1+\left\vert \sin \theta \right\vert \right) }{4\left( 1+2\lambda \right) }-%
\frac{\left\vert \gamma \right\vert ^{2}B_{1}^{2}}{\left( 1+\lambda \right)
^{2}}\cdot \left( 1-\mu -\Re \left( k_{1}\right) \right) .
\end{eqnarray*}
\end{itemize}
Thus the proof is completed.
\end{proof}

\begin{corollary}
\textit{Let }$\gamma =1$ and $\lambda =0.$ If both functions $f$ of the form
(\ref{1}) and its inverse maps $g=f^{-1}$ are in $\mathcal{S}\left[ A,B%
\right] ,$ then using Theorem (\ref{T1}), (\ref{T2}) and (\ref{T3}), we
obtain

$\left( 1\right) $ \ For \ $\gamma \in 
\mathbb{C}
\backslash \left\{ 0\right\} $ and $\mu \in 
\mathbb{C}
,$ 
\begin{equation*}
\ \ \left\vert a_{3}-\mu a_{2}^{2}\right\vert \leq \left\{ 
\begin{array}{c}
\frac{A-B}{2}\text{\ \ \ \ \ \ \ \ \ \ \ \ \ \ \ \ \ \ \ \ \ \ \ \ \ \ \ \ \
\ \ \ \ \ \ \ \ \ \ \ \ \ \ \ \ \ \ \ \ \ if }\left\vert B\right\vert
+\left\vert 4\left( 1-\mu \right) \left( A-B\right) -B\right\vert <2, \\ 
\frac{\left( A-B\right) }{4}\left[ \left\vert B\right\vert +\left\vert
4\left( 1-\mu \right) \left( A-B\right) -B\right\vert \right] \text{\ \ \ \
if \ }\left\vert B\right\vert +\left\vert 4\left( 1-\mu \right) \left(
A-B\right) -B\right\vert \geq 2.%
\end{array}%
\right.
\end{equation*}

$\left( 2\right) $ \ For \ $\gamma >0$ and $\mu \in 
\mathbb{R}
,$ \ \ 
\begin{equation*}
\left\vert a_{3}-\mu a_{2}^{2}\right\vert \leq \left\{ 
\begin{array}{c}
\frac{\left\vert B\right\vert \left( A-B\right) }{2}-\left( \mu -1\right) 
\text{ }\left( A-B\right) ^{2}\text{\ \ \ \ \ \ \ \ \ \ \ \ \ \ \ \ \ \ \ \
\ \ \ \ \ if \ }\mu \leq 1-\frac{1-\left\vert B\right\vert }{2\left(
A-B\right) }, \\ 
\frac{A-B}{2}\text{\ \ \ \ \ \ \ \ \ \ \ \ \ \ \ \ \ \ \ \ \ \ \ \ \ \ \ \ \
\ \ \ \ \ \ \ \ \ \ if \ }1-\frac{1-\left\vert B\right\vert }{2\left(
A-B\right) }<\mu <1+\frac{1-\left\vert B\right\vert }{2\left( A-B\right) },
\\ 
\frac{\left\vert B\right\vert \left( A-B\right) }{2}+\left( \mu -1\right) 
\text{ }\left( A-B\right) ^{2}\text{\ \ \ \ \ \ \ \ \ \ \ \ \ \ \ \ \ \ \ \
\ \ \ \ \ if \ }\mu \geq 1+\frac{1-\left\vert B\right\vert }{2\left(
A-B\right) },%
\end{array}%
\right.
\end{equation*}

\ \ \ \ 

$\left( 3\right) $ \ \ For \ $\gamma \in 
\mathbb{C}
\backslash \left\{ 0\right\} $ and $\mu \in 
\mathbb{R}
,$

\begin{equation*}
\left\vert a_{3}-\mu a_{2}^{2}\right\vert \leq \left\{ 
\begin{array}{c}
\left( A-B\right) ^{2}\text{\ }\left( 1-\mu \right) +\frac{\left\vert
B\right\vert \left( A-B\right) \left( 1+\left\vert \sin \theta \right\vert
-\cos \theta \right) }{4}\text{ \ \ \ \ \ \ \ \ \ \ \ \ \ \ \ \ \ \ \ \ \ \
\ \ \ \ \ \ \ \ if }\mu \leq 1+\psi _{1}\left( A,B,\theta \right) , \\ 
\frac{A-B}{2}\text{ \ \ \ \ \ \ \ \ \ \ \ \ \ \ \ \ \ \ \ \ \ \ \ \ \ \ \ \
\ \ \ \ \ \ \ \ \ \ \ \ \ \ \ \ \ \ \ \ \ \ \ \ \ \ \ \ \ \ \ if }1+\psi
_{1}\left( A,B,\theta \right) {\normalsize <}\mu <1-\psi _{2}\left(
A,B,\theta \right) , \\ 
\frac{\left\vert B\right\vert \left( A-B\right) \left( 1+\left\vert \sin
\theta \right\vert +\cos \theta \right) }{4}-\left( A-B\right) ^{2}\left(
1-\mu \right) \text{\ \ \ \ \ \ \ \ \ \ \ \ \ \ \ \ \ \ \ \ \ \ \ \ \ \ \ \
\ \ \ if \ }\mu \geq 1-\psi _{2}\left( A,B,\theta \right) ,%
\end{array}%
\right.
\end{equation*}

where $B_{1}=A-B$, $B_{2}=-B\left( A-B\right) $, $-1\leq B<A\leq 1,$ $\psi
_{1}\left( A,B,\theta \right) =\frac{\left\vert B\right\vert \left(
1+\left\vert \sin \theta \right\vert -\cos \theta \right) -2}{4\left(
A-B\right) }$ and $\psi _{2}\left( A,B,\theta \right) =\frac{\left\vert
B\right\vert \left( 1+\left\vert \sin \theta \right\vert +\cos \theta
\right) -2}{4\left( A-B\right) }.$
\end{corollary}

\begin{corollary}
\textit{Let }$\gamma =1$ and $\lambda =1.$ If both functions $f$ of the form
(\ref{1}) and its inverse maps $g=f^{-1}$ are in $\mathcal{C}\left[ A,B%
\right] ,$ then using Theorem (\ref{T1}), (\ref{T2}) and (\ref{T3}), we have

$\left( 1\right) $ \ For \ $\gamma \in 
\mathbb{C}
\backslash \left\{ 0\right\} $ and $\mu \in 
\mathbb{C}
,$ \ \ 
\begin{equation*}
\left\vert a_{3}-\mu a_{2}^{2}\right\vert \leq \left\{ 
\begin{array}{c}
\frac{A-B}{6}\text{\ \ \ \ \ \ \ \ \ \ \ \ \ \ \ \ \ \ \ \ \ \ \ \ \ \ \ \ \
\ \ \ \ \ \ \ \ \ \ \ \ \ \ \ \ \ \ \ \ \ if \ }\left\vert B\right\vert
+\left\vert 3\left( 1-\mu \right) \left( A-B\right) -B\right\vert <2, \\ 
\frac{\left( A-B\right) }{12}\left[ \left\vert B\right\vert +\left\vert
3\left( 1-\mu \right) \left( A-B\right) -B\right\vert \right] \text{\ \ \ \
\ if \ }\left\vert B\right\vert +\left\vert 3\left( 1-\mu \right) \left(
A-B\right) -B\right\vert \geq 2.%
\end{array}%
\right.
\end{equation*}

$\left( 2\right) $ \ For \ $\gamma >0$ and $\mu \in 
\mathbb{R}
,$ \ \ 
\begin{equation*}
\left\vert a_{3}-\mu a_{2}^{2}\right\vert \leq \left\{ 
\begin{array}{c}
\frac{\left\vert B\right\vert \left( A-B\right) }{6}-\left( \mu -1\right) 
\frac{\text{ }\left( A-B\right) ^{2}}{4}\text{\ \ \ \ \ \ \ \ \ \ \ \ \ \ \
\ \ \ \ \ \ \ \ \ \ \ \ if \ }\mu \leq 1-\frac{2\left( 1-\left\vert
B\right\vert \right) }{3\left( A-B\right) }, \\ 
\frac{A-B}{6}\text{\ \ \ \ \ \ \ \ \ \ \ \ \ \ \ \ \ \ \ \ \ \ \ \ \ \ \ \ \
\ \ \ \ \ \ \ if \ }1-\frac{2\left( 1-\left\vert B\right\vert \right) }{%
3\left( A-B\right) }<\mu <1+\frac{2\left( 1-\left\vert B\right\vert \right) 
}{3\left( A-B\right) }, \\ 
\frac{\left\vert B\right\vert \left( A-B\right) }{6}+\left( \mu -1\right) 
\frac{\text{ }\left( A-B\right) ^{2}}{4}\text{\ \ \ \ \ \ \ \ \ \ \ \ \ \ \
\ \ \ \ \ \ \ \ \ \ \ \ if \ }\mu \geq 1+\frac{2\left( 1-\left\vert
B\right\vert \right) }{3\left( A-B\right) },%
\end{array}%
\right.
\end{equation*}

$\left( 3\right) $ \ \ For \ $\gamma \in 
\mathbb{C}
\backslash \left\{ 0\right\} $ and $\mu \in 
\mathbb{R}
,$

\begin{equation*}
\left\vert a_{3}-\mu a_{2}^{2}\right\vert \leq \left\{ 
\begin{array}{c}
\text{\ }\left( 1-\mu \right) \frac{\left( A-B\right) ^{2}}{4}+\frac{%
\left\vert B\right\vert \left( A-B\right) \left( 1+\left\vert \sin \theta
\right\vert -\frac{4}{3}\cos \theta \right) }{12}\text{ \ \ \ \ \ \ \ \ \ \
\ \ \ \ \ \ \ \ \ \ \ \ \ \ \ \ \ \ if }\mu \leq 1+\phi _{1}\left(
A,B,\theta \right) , \\ 
\frac{A-B}{6}\text{ \ \ \ \ \ \ \ \ \ \ \ \ \ \ \ \ \ \ \ \ \ \ \ \ \ \ \ \
\ \ \ \ \ \ \ \ \ \ \ \ \ \ \ \ \ \ \ \ \ \ \ \ \ \ \ \ if }1+\phi
_{1}\left( A,B,\theta \right) {\normalsize <}\mu <1-\phi _{2}\left(
A,B,\theta \right) , \\ 
\frac{\left\vert B\right\vert \left( A-B\right) \left( 1+\left\vert \sin
\theta \right\vert +\frac{4}{3}\cos \theta \right) }{12}-\left( 1-\mu
\right) \frac{\left( A-B\right) ^{2}}{4}\text{\ \ \ \ \ \ \ \ \ \ \ \ \ \ \
\ \ \ \ \ \ \ \ \ \ \ \ \ \ if \ }\mu \geq 1-\phi _{2}\left( A,B,\theta
\right) ,%
\end{array}%
\right.
\end{equation*}

where $B_{1}=A-B$, $B_{2}=-B\left( A-B\right) $, $-1\leq B<A\leq 1,$ $\phi
_{1}\left( A,B,\theta \right) =\frac{\left\vert B\right\vert \left(
1+\left\vert \sin \theta \right\vert -\cos \theta \right) -2}{3\left(
A-B\right) }$ and $\phi _{2}\left( A,B,\theta \right) =\frac{\left\vert
B\right\vert \left( 1+\left\vert \sin \theta \right\vert +\cos \theta
\right) -2}{3\left( A-B\right) }.$
\end{corollary}

\begin{corollary}
Let $\gamma \in 
\mathbb{C}
\backslash \left\{ 0\right\} $ and $\lambda =0.$ If both functions $f$ of
the form (\ref{1}) and its inverse maps $g=f^{-1}$ are in $\mathcal{S}^{\ast
}[\gamma ],$ then similarly, using Theorem (\ref{T1}), (\ref{T2}) and (\ref%
{T3}), we obtain
$\left( i\right) $ \ For \ $\gamma \in 
\mathbb{C}
\backslash \left\{ 0\right\} $ and $\mu \in 
\mathbb{C}
,$ \ \ 
\begin{equation*}
\left\vert a_{3}-\mu a_{2}^{2}\right\vert \leq \left\{ 
\begin{array}{c}
\left\vert \gamma \right\vert \text{\ \ \ \ \ \ \ \ \ \ \ \ \ \ \ \ \ \ \ \
\ \ \ \ \ \ \ \ \ \ \ \ \ \ \ \ \ \ \ \ \ if \ }\left\vert 1+\left( 1-\mu
\right) 8\gamma \right\vert <1, \\ 
\frac{\left\vert \gamma \right\vert }{2}\text{\ }\left[ \left\vert 1+\left(
1-\mu \right) 8\gamma \right\vert +1\right] \text{\ \ \ \ \ \ \ if \ }%
\left\vert 1+\left( 1-\mu \right) 8\gamma \right\vert \geq 1.%
\end{array}%
\right.
\end{equation*}

$\left( ii\right) $ \ For \ $\gamma >0$ and $\mu \in 
\mathbb{R}
,$ \ \ 
\begin{equation*}
\left\vert a_{3}-\mu a_{2}^{2}\right\vert \leq \left\{ 
\begin{array}{c}
\gamma -4\left( \mu -1\right) \gamma ^{2}\text{\ \ \ \ \ \ if \ }\mu \leq 1,
\\ 
\gamma +4\left( \mu -1\right) \gamma ^{2}\text{\ \ \ \ \ \ \ if \ }\mu >1,%
\end{array}%
\right.
\end{equation*}

$\left( iii\right) $ \ \ For \ $\gamma \in 
\mathbb{C}
\backslash \left\{ 0\right\} $ and $\mu \in 
\mathbb{R}
,$

\begin{equation*}
\left\vert a_{3}-\mu a_{2}^{2}\right\vert \leq \left\{ 
\begin{array}{c}
\text{\ }4\left\vert \gamma \right\vert ^{2}\left( 1-\mu \right) +\frac{%
\left\vert \gamma \right\vert \left( 1+\left\vert \sin \theta \right\vert
-\cos \theta \right) }{2}\text{ \ \ \ \ \ \ \ \ \ \ \ \ \ \ \ \ \ \ \ \ \ \
\ \ \ \ if }\mu \leq 1+\psi _{1}\left( \gamma ,\theta \right) , \\ 
\left\vert \gamma \right\vert \text{\ \ \ \ \ \ \ \ \ \ \ \ \ \ \ \ \ \ \ \
\ \ \ \ \ \ \ \ \ \ \ \ \ \ \ \ \ \ \ \ \ \ \ \ \ \ \ \ if }1+\psi
_{1}\left( \gamma ,\theta \right) {\normalsize <}\mu <1-\psi _{2}\left(
\gamma ,\theta \right) , \\ 
\frac{\left\vert \gamma \right\vert \left( 1+\left\vert \sin \theta
\right\vert -\cos \theta \right) }{2}-4\left\vert \gamma \right\vert
^{2}\left( 1-\mu \right) \text{\ \ \ \ \ \ \ \ \ \ \ \ \ \ \ \ \ \ \ \ \ \ \
\ \ \ \ if \ }\mu \geq 1-\psi _{2}\left( \gamma ,\theta \right) ,%
\end{array}%
\right.
\end{equation*}

where $B_{1}=2$, $B_{2}=2$, $\psi _{1}\left( \gamma ,\theta \right) =\frac{%
\left( \left\vert \sin \theta \right\vert -\cos \theta -1\right) }{%
8\left\vert \gamma \right\vert \text{\ }}$ and $\psi _{2}\left( \gamma
,\theta \right) =\frac{\left( \left\vert \sin \theta \right\vert +\cos
\theta -1\right) }{8\left\vert \gamma \right\vert \text{\ }}$.
\end{corollary}

\begin{corollary}
\textit{Let } $\gamma \in 
\mathbb{C}
\backslash \left\{ 0\right\} \ $and $\lambda =1.$ Let both functions $f$ of
the form (\ref{1}) and its inverse maps $g=f^{-1}$ are in $\mathcal{C}%
[\gamma ].$ Then similarly, using Theorem (\ref{T1}), (\ref{T2}) and (\ref%
{T3}), we have

$\left( i\right) $ \ For \ $\gamma \in 
\mathbb{C}
\backslash \left\{ 0\right\} $ and $\mu \in 
\mathbb{C}
,$ \ \ 
\begin{equation*}
\left\vert a_{3}-\mu a_{2}^{2}\right\vert \leq \left\{ 
\begin{array}{c}
\frac{\left\vert \gamma \right\vert }{3}\text{\ \ \ \ \ \ \ \ \ \ \ \ \ \ \
\ \ \ \ \ \ \ \ \ \ \ \ \ \ if \ }\left\vert 1+\left( 1-\mu \right) 6\gamma
\right\vert <1, \\ 
\frac{\left\vert \gamma \right\vert }{2}\text{\ }\left[ \left\vert 1+\left(
1-\mu \right) 6\gamma \right\vert +1\right] \text{\ \ \ \ \ \ \ if \ }%
\left\vert 1+\left( 1-\mu \right) 6\gamma \right\vert \geq 1.%
\end{array}%
\right.
\end{equation*}

$\left( ii\right) $ \ For \ $\gamma >0$ and $\mu \in 
\mathbb{R}
,$ \ \ 
\begin{equation*}
\left\vert a_{3}-\mu a_{2}^{2}\right\vert \leq \left\{ 
\begin{array}{c}
\frac{\gamma }{3}-\left( \mu -1\right) \gamma ^{2}\text{\ \ \ \ \ \ if \ }%
\mu \leq 1, \\ 
\frac{\gamma }{3}+\left( \mu -1\right) \gamma ^{2}\text{\ \ \ \ \ \ \ if \ }%
\mu >1,%
\end{array}%
\right.
\end{equation*}

$\left( iii\right) $ \ \ For \ $\gamma \in 
\mathbb{C}
\backslash \left\{ 0\right\} $ and $\mu \in 
\mathbb{R}
,$

\begin{equation*}
\left\vert a_{3}-\mu a_{2}^{2}\right\vert \leq \left\{ 
\begin{array}{c}
\text{\ }\left\vert \gamma \right\vert ^{2}\left( 1-\mu \right) +\frac{%
\left\vert \gamma \right\vert \left( 1+\left\vert \sin \theta \right\vert
-\cos \theta \right) }{6}\text{ \ \ \ \ \ \ \ \ \ \ \ \ \ \ \ \ \ \ \ \ \ \
\ \ \ \ if }\mu \leq 1+\phi _{1}\left( \gamma ,\theta \right) , \\ 
\frac{\left\vert \gamma \right\vert }{3}\text{\ \ \ \ \ \ \ \ \ \ \ \ \ \ \
\ \ \ \ \ \ \ \ \ \ \ \ \ \ \ \ \ \ \ \ \ \ \ \ \ \ \ \ \ \ \ \ if }1+\phi
_{1}\left( \gamma ,\theta \right) {\normalsize <}\mu <1-\phi _{2}\left(
\gamma ,\theta \right) , \\ 
\frac{\left\vert \gamma \right\vert \left( 1+\left\vert \sin \theta
\right\vert -\cos \theta \right) }{6}-\left\vert \gamma \right\vert
^{2}\left( 1-\mu \right) \text{\ \ \ \ \ \ \ \ \ \ \ \ \ \ \ \ \ \ \ \ \ \ \
\ \ \ \ if \ }\mu \geq 1-\phi _{2}\left( \gamma ,\theta \right) ,%
\end{array}%
\right.
\end{equation*}

where $B_{1}=2$, $B_{2}=2$, $\phi _{1}\left( \gamma ,\theta \right) =\frac{%
\left( \left\vert \sin \theta \right\vert -\cos \theta -1\right) }{%
6\left\vert \gamma \right\vert \text{\ }}$ and $\phi _{2}\left( \gamma
,\theta \right) =\frac{\left( \left\vert \sin \theta \right\vert +\cos
\theta -1\right) }{6\left\vert \gamma \right\vert \text{\ }}$.

\begin{acknowledgement}
The research of E. Deniz and M. Ça\u{g}lar were supported by the Commission
for the Scientific Research Projects of Kafkas University, project number
2016-FM-67.
\end{acknowledgement}
\end{corollary}

\end{document}